\documentclass[11pt]{article}
\usepackage{amssymb}
\usepackage{graphicx}
\usepackage{xcolor} 
\usepackage{tensor}
\usepackage{fullpage} 
\usepackage{amsmath}
\usepackage{amsthm}
\usepackage{verbatim}
\usepackage{enumitem}
\setlist[enumerate]{leftmargin=1.5em}
\setlist[itemize]{leftmargin=1.5em}
\usepackage{hyperref}
\usepackage{seqsplit,mathtools}

\definecolor{green}{rgb}{0,0.8,0} 

\newtheorem{theorem}{Theorem}[section]
\newtheorem{thm}{Theorem}[section]
\newtheorem{corollary}[theorem]{Corollary}
\newtheorem{lemma}[theorem]{Lemma}
\newtheorem{proposition}[theorem]{Proposition}

\theoremstyle{definition}

\theoremstyle{remark}
\newtheorem{remark}[theorem]{Remark}

\numberwithin{equation}{section}
\newcommand{\nrm}[1]{\Vert#1\Vert}

\newcommand{\nnrm}[1]{{\vert\kern-0.25ex\vert\kern-0.25ex\vert #1 
		\vert\kern-0.25ex\vert\kern-0.25ex\vert}}

\newcommand{\supp}{{\mathrm{supp}}\,}

\newcommand{\lap}{\Delta}

\newcommand{\rd}{\partial}
\newcommand{\nb}{\nabla}

\newcommand{\alp}{\alpha}

\newcommand{\dlt}{\delta}

\newcommand{\varep}{\varepsilon}

\newcommand{\tht}{\theta}

\newcommand{\omg}{\omega}
\newcommand{\Omg}{\Omega}

\newcommand{\bbR}{\mathbb R}

\newcommand{\bbT}{\mathbb T}

\begin{document}
\bibliographystyle{plain}
 \title{
{Infinite growth in vorticity gradient of compactly supported \\
planar vorticity near  Lamb dipole}}
\author{Kyudong Choi\thanks{Department of Mathematical Sciences, Ulsan National Institute of Science and Technology. \newline
		E-mail: kchoi@unist.ac.kr} 
	\and In-Jee Jeong\thanks{Department of Mathematical Sciences and RIM, Seoul National University. \newline \qquad  E-mail: injee\_j@snu.ac.kr}
}

\date\today

\maketitle

\renewcommand{\thefootnote}{\fnsymbol{footnote}}


\renewcommand{\thefootnote}{\arabic{footnote}}

\begin{abstract}  
We prove linear in time filamentation for perturbations of the Lamb dipole, which is a traveling wave solution to the  incompressible Euler equations in $\bbR^2$. The main ingredient is a recent nonlinear orbital stability result by Abe--Choi. As a consequence, we obtain linear in time growth for the vorticity gradient for all times, for certain smooth and compactly supported initial vorticity in $\bbR^2$. The construction carries over to some generalized SQG equations.  
\end{abstract}\vspace{1cm}

\section{Introduction}
\subsection{Main results}

In this paper, we are concerned with dynamics of the two-dimensional incompressible Euler equations near \textit{Lamb's vortex dipole}. Recall that in terms of the vorticity, the 2D Euler equation is given by \begin{equation}\label{2d_euler}
	\begin{aligned}
		\partial_t \omega+u\cdot \nabla \omega=0,\quad  u&=\mathcal{K}[\omega]=K*\omega \quad \textrm{in}\quad \mathbb{R}^{2}\times (0,\infty), 
	\end{aligned}
\end{equation}
with the 2D Biot--Savart kernel $K(x)=(2\pi)^{-1}x^{\perp}|x|^{-2}$, $x^{\perp}=
(-x_2,x_1)$. The Lamb dipole {(or Chaplygin--Lamb dipole) is a traveling wave solution of the 2D Euler equations
introduced  by 
 H. Lamb \cite[p231]{Lamb} in 1906 and,  independently, by S. A. Chaplygin in 1903 \cite{Chap1903}, \cite{Chap07} 
  (see \cite{MV94} for related history).} This solution is explicitly given by \begin{equation}   \label{lamb}
	\omega_L=\begin{cases} 
		g(r)\sin\theta,\quad & r\leq 1,\\
		0 ,\quad & r>1,\end{cases}
\end{equation} in polar coordinates $x_1=r\cos\theta$ and $x_2=r\sin\theta$, where we set 
\begin{equation*}
	g(r):=\left(\frac{-2c_L}{J_0(c_L)}\right)J_1(c_L r).
\end{equation*} 
Here, $J_{m}(r)$ is the $m$-th order Bessel function of the first kind. The constant $c_L=3.8317\dots>0$ is the first (positive) zero point of $J_1$ and  $J_0(c_L)<0$.  Then, one can check that 
$$\omega(t,x)=\omega_L(x-W_Lte_{x_1}),\quad W_L:=1$$ is
an odd-symmetric (with respect to $x_1$-axis) solution of \eqref{2d_euler}, see Subsection \ref{subsec:Lamb} for details. 

\medskip

\noindent Our first main result is the following {dynamical} stability theorem for the Lamb dipole.  We shall denote $L^p:=L^p(\mathbb{R}^2_+)$, $\|f\|_{L^1\cap L^2}:=\|f\|_{L^1}+\|f\|_{L^2}$, and $\|x_2 f\|_{L^1}:=\int_{\mathbb{R}^2_+} x_2|f(x)|\,dx$.
\begin{thm}[Quantitative stability]\label{thm_lamb_lower} 
	(I) For   $\varepsilon>0$,      there exists $\delta >0$ such that the following holds. For $\omg_{0}$ satisfying \begin{itemize}
		\item $\omega_0\in L^{1}\cap {L^{\infty}}$ and $x_2\omega_{0}\in L^{1}$,
		\item $\omega_0\geq 0$ on $\mathbb{R}^2_+$ and $\omg_0$ is odd-symmetric with respect to $x_1$-axis, 
		\item $\left\|\omega_0-\omega_{L} \right\|_{L^1\cap L^2}+\left\|x_2(\omega_0-\omega_{L}) \right\|_{L^1}\leq \delta,$
	\end{itemize} there exists a function $\tau:[0,\infty)\to\mathbb{R}$ with $\tau(0)=0$  satisfying 
	\begin{align}\label{lamb_stab}
		\left\|\omega(t)-\omega_{L}(\cdot-\tau(t)e_{x_1}) \right\|_{L^1\cap L^2}+\left\|x_2(\omega(t)-\omega_{L}(\cdot-\tau(t)e_{x_1})) \right\|_{L^1} \leq \varepsilon,\quad  t\geq0, 
	\end{align}
	where $\omega(t)$ is the unique weak solution of \eqref{2d_euler} for the initial data $\omega_0$.
	
	(II)   There exists a constant $\varepsilon_0>0$  such that the following holds. For any $M>0$, there exists $C_M>0$ such that if $\omega_0$ satisfies
	$$\nrm{\omega_0}_{L^\infty}\leq M $$ in addition to the hypotheses from $(I)$ with some $\dlt=\dlt(\varep)$ and $0<\varepsilon<\varepsilon_0$, then the shift function $\tau$ from \eqref{lamb_stab} satisfies
	\begin{equation}\label{lamb_main_est} 
		|\tau(t) -W_Lt|\leq C_M(t+1)\varepsilon^{1/3},\quad t\geq0.
	\end{equation}  	
\end{thm}
We  do not claim any originality of
the orbital stability statement (I)   above, which  is just a minor modification of \cite[Theorem 1.1]{AC2019}. Instead, our contribution lies on the estimate \eqref{lamb_main_est} of the shift function $\tau$. The condition  $\omg_0 \in L^\infty$ in (I) of the above  was inserted simply to guarantee uniqueness of the solution. Based on the above stability result, we are able to prove linear in time \textit{filamentation} 
for perturbations of the Lamb dipole.  The precise statement will be understood during the proof of the theorem below on the linear in time growth of the gradient and support for smooth solutions to \eqref{2d_euler}.
\begin{thm}[{Instability:} gradient and support growth]\label{thm_fil_lamb}
	There exists a $C^{\infty}$--smooth and compactly supported vorticity $\omg_{0}$  in $\mathbb{R}^2$ such that $\omg_0$ is odd-symmetric with respect to $x_1$-axis,
$\omega_0\geq 0$ on $\mathbb{R}^2_+$, and the corresponding solution $\omega$ satisfies 
	\begin{equation*}
		\begin{split}
			\inf_{p\in[1,\infty]}\nrm{\nb\omg(t,\cdot)}_{L^{p}} \ge c_{0}t, \quad \nrm{\omg(t,\cdot)}_{C^{\alp} } \ge c_{\alp}t^{\alp}, \quad  \mathrm{diam}( \mathrm{supp} (\omg(t,\cdot)) \cap \mathbb{R}^2_+) \ge c_{0}t, \quad \forall t\geq t_0
		\end{split}
\end{equation*} for some universal constants $t_0, c_{0}>0$ and $c_{\alp}>0$ for any $0<\alp\le 1$.  \end{thm} 
From the literature, we are not aware of any results giving infinite growth of the vorticity gradient in the whole space  $\mathbb{R}^2$, with smooth and compactly supported initial data. Let us emphasize that if at least one of the following restrictions is lifted, there are previous works which show gradient growth for all times: (i) unbounded domain, (ii) $C^{\infty}_c$--data, and (iii) absence of physical boundaries. 
It is expected that the solution we construct actually satisfies $\nrm{\omg(t,\cdot)}_{W^{s,1}} \gtrsim_{s} t^{s}$ for any $0<s<1$, which would confirm in a precise sense that the $L^{\infty}$--norm is the strongest globally controlled quantity for smooth vorticities. 

\subsection{Extensions to active scalar equations}

The stability statement as well as the infinite-time linear growth of the gradient carries over to the case of certain \textit{active scalar} equations, which is known to possess traveling wave solutions similar to the Lamb dipole. Indeed, on $\bbR^{2}$, one can consider the following family of systems, which is often referred to as $\alp$--SQG equations: \begin{equation}\label{eq:alp-SQG}
	\left\{
	\begin{aligned}
	&\rd_t\tht + u\cdot\nb \tht = 0,	 \\
	&u = -\nb^\perp (-\lap)^{-\alp}\tht. 
	\end{aligned}
	\right.
\end{equation} Here, $\tht(t,\cdot):\bbR^2\rightarrow\bbR$ and the operator $ (-\lap)^{-\alp}$ is defined by the convolution against $C_\alp|\cdot|^{-(2-2\alp)}$ for some $C_\alp>0$. The case $\alp=1$ is simply the 2D Euler equations in vorticity form, and $\alp=\frac{1}{2}$ corresponds to the surface quasi-geostrophic (SQG) equation introduced in \cite{CMT1}. In the range $0 < \alp < 1$, local well-posedness is well known for smooth initial data (\cite{CCCGW}), but the problem of finite time singularity formation remains open. (However, see \cite{KYZ,KRYZ,GaPa} and references therein.) Existence and orbital stability of dipole structures are known for the entire range $0<\alp<1$, thanks to the very recent work \cite{CQZZ-stab}. Based on the stability statement, we can prove the following result on gradient growth. 
\begin{proposition}\label{prop:sqg}
	For $\frac12<\alp<1$, there exists a $C^{\infty}$--smooth and compactly supported data $\tht_{0}$ in $\bbR^{2}$ such that the associated local unique smooth solution  satisfies the growth \begin{equation*}
		\begin{split}
			\inf_{p\in[1,\infty]}\nrm{\nb\tht(t,\cdot)}_{L^{p}} \ge c_{0}t
		\end{split}
	\end{equation*} for all $t \in [0,T^*)$, where $T^* \in (0,\infty]$ is the lifespan of the solution with data $\tht_{0}$. 
\end{proposition} In this result, we cannot rule out the possibility of singularity formation, similarly as in the works \cite{KN,HeKi}. Moreover, note that unfortunately the SQG equation is excluded from the above. This is simply because the orbital stability holds in $L^p$--norms of $\tht$ with $p<\infty$, which is not sufficient to control the velocity in $L^{\infty}$ starting from the SQG case. We refer the interested reader to the introduction of \cite{HeKi} for a further discussion. It will be an interesting problem to obtain an instability similar to \ref{prop:sqg} in the range $0< \alp < \frac12$. 

\subsection{Previous works} While there is a vast literature on theoretical study of the two-dimensional Euler equations and more generally on active scalars, let us briefly review the works which are the most closely related to our main results. 

\medskip 

\noindent \textbf{Results for the Lamb dipole}. Being an explicit traveling wave solution for two dimensional incompressible flows, the Lamb dipole has served as a popular model of vortex motion. Indeed, vortex structures resembling Lamb dipole are frequently observed in laboratory experiments \cite{CoBa,FvH,BiCho}. Moreover, the Lamb dipole provides a nice benchmark problem for testing the stability of numerical schemes, since its vorticity has a highly non-trivial profile. Extensive numerical studies have been carried out for perturbations of the Lamb dipole, clearly showing filamentation (formation of long arms) with arbitrarily small initial perturbations (\cite{Niel1,Niel2,GeHe,KX}). Such an instability is rigorously confirmed in the current work. 

\medskip

\noindent \textbf{Gradient growth for active scalars}. Note that smooth solutions to the equations \eqref{2d_euler} and \eqref{eq:alp-SQG} conserve all $L^p$--norms in time $(1\le p \le\infty)$. Although it is a natural question to ask whether there exists a smooth solution experiencing large gradient growth, the problem has remained essentially open (except for linear growth results using physical boundaries \cite{Nad,Y3}) before the pioneering works of Denisov \cite{Den,Den2,Denisov-merging} and Kiselev--Nazarov \cite{KN}. The strategy in \cite{Den} gives superlinear in time gradient growth could be extended to a range of $\alp$--SQG equations. On the other hand, the result of \cite{KN} is applicable to all active scalar equations but the growth is in terms of some high Sobolev norm of $\tht$ and available only for a finite time horizon. Then, Kiselev--Sverak proved double exponential growth for the vorticity gradient in a disc (see also \cite{Xu}), which is the sharp rate. Using a similar configuration, Zlatos obtained exponential in time growth of $\nrm{\nb\omg(t,\cdot)}_{L^{\infty}}$ in $\bbT^2$ (\cite{Z}), but it is interesting to note that the initial vorticity is less smooth than $C^{1,1}(\bbT^2)$; that is, smoothness of the vorticity is an obstruction to exponential in time gradient growth. Very recently, He--Kiselev constructed smooth initial data experiencing exponential in time gradient growth for the SQG equation for $\bbT^{2}$ in \cite{HeKi}, and this proof could be extended to the entire range $0<\alp<1$. Note that in the aforementioned works, the physical domain was bounded. In the case of $\bbR^2$, obtaining the gradient growth is strictly harder; mainly, one needs to deal with the possibility of vorticity simply escaping to infinity with time. Our previous works on perturbations of the circular vortex patch (\cite{Choi2019,CJ,CJ_windin}) could be adapted to give gradient growth for smooth initial data close to the circular patch, but only for arbitrary large finite time horizons.

\medskip

\noindent \textbf{Support diameter growth for vortex patches}. A closely related question is whether there exists a \textit{vortex patch} solution for \eqref{2d_euler} which experiences growth in time of the diameter. By a vortex patch, we mean a solution to \eqref{2d_euler} which is of the form $\omg(t,x)=\mathbf{1}_{\Omg(t)}$, i.e. the indicator function of a moving domain $\Omg(t)$. It is a notoriously difficult open problem to construct such a solution satisfying $\mathrm{diam}(\Omg(t))\to\infty$ as $t\to \infty$, even when one allows the set $\Omg_{0}$ to be disconnected. On the other hand, there are such examples when the patch is allowed to have mixed sign (\cite{ISG99,Zb}); $\omg(t,x) = \sum_{i=1}^{N} a_i \mathbf{1}_{\Omg_i(t)}$ for some $a_i\in\bbR$. The work \cite{ISG99} used patches in $\bbR^2$ which are odd with respect to both axes and proved linear in time support growth. This work can be interpreted as a stability result for the odd-odd quadruple of point vortices, whose dynamics is explicitly solvable. A very recent breakthrough by Zbarsky \cite{Zb} achieved stability in the class containing patches for certain self-similar expanding configuration of three point vortices, in particular obtaining growth of the path diameter with order $t^{1/2}$. This is striking because in this configuration, the vorticity does not need to satisfy any symmetry condition. However, it is not clear to us whether such stability results from \cite{ISG99,Zb} are able to produce infinite in time growth of the vorticity gradient for nearby smooth solutions. For further results on small scale creation for patches, see \cite{Ki-sur,Ki-sur2,EJSVP1,EJSVP2,KRYZ,KiLi}.

\medskip

\noindent \textbf{Axisymmetric Euler equations and Hill's vortex}. The questions of vorticity gradient and support diameter growth could be asked for the incompressible three-dimensional Euler equations as well. In particular, one may consider the axisymmetric Euler equations without swirl, which share some structures with the two-dimensional Euler equations. While global regularity for smooth initial data is known (\textit{e.g. } see \cite{MB}), there are only a few results on the norm growth for smooth solutions; see references in \cite{CJ_Hill,Do}. In our companion work \cite{CJ_Hill} (based on orbital stability \cite{Choi2020}), we consider perturbations of the Hill's vortex, which is a traveling--wave solution to the axisymmetric Euler equation supported on the unit ball in $\bbR^3$. The Hill's vortex could be seen as the 3D analogue of the Lamb dipole. In \cite{CJ_Hill}, stability and instability of the Hill's vortex is established using a largely parallel argument as in this work. However, let us point out the main differences: the gradient growth in the current paper is strictly harder to achieve since we face the restriction that the vorticity must vanish on the symmetry axis $\{ x_{2} = 0 \}$, for the vorticity to be smooth.  {In addition, the Lamb dipole $\omega_L$ vanishes on the boundary of its supporting disk, which obligates us to prove a weaker  estimate (Lemma \ref{lem:lamb_error}).  }
On the other hand, in the case of Hill's vortex, the advected scalar (which is the axial vorticity divided by the distance to the symmetry axis) does not need to vanish on the axis. In this latter case, gradient growth is immediate from diameter growth along the axis. Moreover, the instability mechanism near the Hill's vortex only gives infinite time growth of the vorticity Hessian. \\

Lastly, we remark that the proof of \eqref{lamb_main_est} might become  easier when   we consider only the initial data $\omega_0$  supported in $\{|x_1|\leq K\}$ for any given  constant $K>0$ (even if the resulting estimate  depends on the parameter $Z_0$). Indeed, we can  use the identity on the the speed of the center of mass in $x_1$-direction:
	$$
	\frac{d}{dt}\int_{\mathbb{R}^2_+} x_1\omega (t,x)\,dx=\int_{\mathbb{R}^2_+} u^{x_1}(t,x)\omega (t,x)\,dx.
	$$   
Indeed, the center moves  linearly in the direction of the symmetry axis (see \cite[Sec. 5.1]{Iftimie_lec}).

{ \begin{remark}
		Recently, we have learned about a new and general Lagrangian stability result for 2D Euler from T. Drivas and T. Elgindi (\cite{DE}, private communication). It seems that their result applies for perturbations of the Lamb dipole and is able to reproduce Theorem \ref{thm_fil_lamb}. 
	\end{remark}
}

\section{Proof}

This section is organized as follows. To begin with, in Subsection \ref{subsec:Lamb} we provide some detailed properties of the Lamb dipole. Then in Subsection 2.2, we provide a reformulation of the orbital stability statement from \cite{AC2019}. Theorem \ref{thm_lamb_lower} is proved in Subsection 2.4, after some preliminary lemmas in Subsection 2.3. Finally, we prove Theorem \ref{thm_fil_lamb} in Subsection 2.5. While the proof of Proposition \ref{prop:sqg} is completely parallel to that for 2D Euler, we sketch the main steps in Subsection 2.6.

\subsection{Properties of   Lamb dipole}\label{subsec:Lamb}

For the reader's convenience, let us discuss  in some detail   properties of the Lamb dipole. Recall that 
\begin{equation*}
	\omega_L=\begin{cases} 
		g(r)\sin\theta,\quad &r\leq 1,\\
		0 ,\quad &r>1\end{cases}
\end{equation*} and \begin{equation*}
	\begin{split}
		g(r)=\left(\frac{-2c_L}{J_0(c_L)}\right)J_1(c_L r).
	\end{split}
\end{equation*} Let us recall that the Bessel function
$J_i:\mathbb{R}_{\geq0}\to\mathbb{R},\quad i\in\mathbb{Z}$ of the first kind is defined by the solution of 
$$r^2J_i''(r)+rJ_i'(r)+(r^2-i^2)J_i(r)=0\quad\mbox{and}\quad  J_i(0)\quad \mbox{is finite}.$$ A direct computation using the form of the Laplacian in polar coordinates gives that $\psi_{L}$ defined by 
\begin{equation}   
	\psi_L =
	\begin{cases} (c_L^{-2}g(r)+r)\sin\theta, &\quad r\leq 1,\\
		\frac{1}{r}\sin\theta ,\quad & r>1\end{cases}
\end{equation} satisfies $-\lap \psi_{L} = \omg_{L}.$ Since $\psi_{L}$ decays at infinity, we see that $\psi_{L}$ is the stream function associated with $\omg_L$; we have $\psi_L=(-\Delta_{\mathbb{R}^2})^{-1}\omega_L=-\frac{1}{2\pi}\log|\cdot_x|*_{\mathbb{R}^2}\omega_L.$ The corresponding velocity field $u_L$ can be obtained by
$$u_L=-\nabla^\perp \psi_L=(\partial_{x_2}\psi_L,-\partial_{x_1}\psi_L).$$
In polar coordinates, we have explicitly 
\begin{equation}  \label{eq:vel1}
	u_L^{1}=\begin{cases}
		1+\frac{g(r)}{c_L^2 r}+\frac{\sin^2\theta}{c_L^2 r}[rg'(r)-g(r)]
		\quad & r\leq 1,\\
		\frac{1}{r^2}\cos 2\theta ,\quad & r>1,\end{cases}
\end{equation} and
\begin{equation*}   
	u_L^{2}=\begin{cases} \frac {-1} {2c_L^2 r} \sin 2\theta \left[
		rg'(r)-g(r)
		\right],\quad & r\leq 1,\\
		\frac{1}{r^2}\sin 2\theta ,\quad &r>1.\end{cases}
\end{equation*} Now, to see that $\omg_L$ indeed defines a traveling wave solution, one can verify that
\begin{equation*}\label{lamb_dipole_vor}
	\omega_{L}=c_L^2 f_L(\psi_L-W_L x_2)\quad\mbox{in}\quad\mathbb{R}^2_+:=\{x\in\mathbb{R}^2\,|\, x_2> 0\}, 
\end{equation*} where $f_L$ is simply 
\begin{equation*}
	f_L(s)=s^+=\begin{cases} s,\quad &s>0\\ 0,\quad &s\leq 0.\end{cases}
\end{equation*} This functional relationship between $\omg_L$ and $\psi_L - W_Lx_2$ shows that $\omg_L$ defines a traveling wave solution with velocity $W_L e_{x_1}$. 

Let us now remark on  regularity of the Lamb dipole. We note that $g\in C^2([0,\infty))$ and satisfies
$$g(0)=g(1)=0,\quad g'(1)<0,\quad  \mbox{and} \quad g(r)>0,\quad r\in(0,1).$$ This shows that $\omg_{L} \in { C^{0,1}(\bbR^2)}$. Then, from Schauder theory, $\psi_L \in C^{2,\alpha}(\mathbb{R}^2)$ for any $\alpha\in(0,1)$. In particular, the velocity $u_L:=\mathcal{K}[\omega_L]$ is Lipschitz continuous.

\subsection{Orbital stability in the norm $(L^1+L^2+\mbox{impulse})$}
 
We borrow   the stability theorem from \cite[Theorem 1.1]{AC2019} adapted to the Lamb dipole   on the \textit{unit} disk with \textit{unit} speed \eqref{lamb}:
\begin{thm}[{{\cite[Theorem 1.1]{AC2019}}}]\label{thm_lamb}
  The Lamb dipole $\omega_{L}$ is orbitally stable in the following sense: for $\nu>0$ and $\varepsilon>0$, there exists $\delta=\delta(\nu,\varepsilon)>0$ such that for $\omg_0$ satisfying \begin{itemize}
  	\item $\omega_0\in L^{1}\cap L^{2}(\mathbb{R}^{2})$, $x_2\omega_{0}\in L^{1} $,  
  	\item $\omega_0\geq 0$ on $\mathbb{R}^2_+$ and $\omega_0$ is {odd-symmetric with respect to the $x_1$-axis},
  	\item $\nrm{\omega_0}_{L^1}\leq \nu$, and
  	\item $\left\|\omega_0-\omega_{L} \right\|_{  L^2}+\left\|x_2(\omega_0-\omega_{L}) \right\|_{L^1}\leq \delta$,
  \end{itemize} 
there exists a global weak solution $\omega(t)$ of \eqref{2d_euler} satisfying
\begin{align*}
\inf_{\tau\in\mathbb{R}}\left\{\left\|\omega(t)-\omega_{L}(\cdot-\tau e_{x_1}) \right\|_{ L^2}+\left\|x_2(\omega(t)-\omega_{L}(\cdot-\tau e_{x_1})) \right\|_{L^1}\right\}\leq \varepsilon,\quad \textrm{for all}\ t\geq0.  
\end{align*} 
\end{thm}
We note that the stability norm above is in terms of $(L^2+\mbox{impulse})$.  
For convenience,  we will use the following modification
using $(L^1+L^2+\mbox{impulse})$, since $L^1$--stability in vorticity is more natural   when controlling $L^\infty$ of the fluid velocity. To this end we present the following result. 
\begin{corollary}\label{cor_lamb}
For  $\varepsilon>0$, there exists $\delta=\delta(\varepsilon)>0$ such that for $\omega_0\in L^{1}\cap L^{2}(\mathbb{R}^{2})$ satisfying $x_2\omega_{0}\in L^{1}$,  
 $$\omega_0\geq 0 \quad\mbox{on}\quad \mathbb{R}^2_+,\quad\omega_0:\,\mbox{odd-symmetric with respect to $x_1$-axis},$$
   and 
\begin{align*}
 \left\|\omega_0-\omega_{L} \right\|_{L^1\cap L^2}+\left\|x_2(\omega_0-\omega_{L}) \right\|_{L^1}\leq \delta,
\end{align*}
there exist a global weak solution $\omega(t)$ of  \eqref{2d_euler}
and  a function $\tau:[0,\infty)\to\mathbb{R}$ satisfying $\tau(0)=0$ and
\begin{align}\label{assum_lem_stab}
\sup_{t\geq 0}\left\{\left\|\omega(t)-\omega_{L}(\cdot-\tau(t) e_{x_1}) \right\|_{ L^1\cap L^2}+\left\|x_2(\omega(t)-\omega_{L}(\cdot-\tau(t) e_{x_1})) \right\|_{L^1}\right\}\leq \varepsilon. 
\end{align} 
\end{corollary}
\begin{proof}
 We denote $$\nu_L:= \int_{\mathbb{R}^{2}_{+}}\omega_L dx\in(0,\infty).$$
Let $\varepsilon>0$. From Theorem \ref{thm_lamb},
 we take a constant $$\delta=\delta(\nu_L+1,\frac{\varepsilon}{16})>0.$$ We may assume $\delta<1$ and $\delta<\varepsilon/4$. Then consider any initial data
  $\omega_0\in L^{1}\cap L^{2}(\mathbb{R}^{2}_{+})$ satisfying $x_2\omega_{0}\in L^{1}(\mathbb{R}^{2}_{+})$, $\omega_{0}\geq 0$
   and 
\begin{align*}
 \left\|\omega_0-\omega_{L} \right\|_{ L^1\cap L^2}+\left\|x_2(\omega_0-\omega_{L}) \right\|_{L^1}\leq \delta. 
\end{align*}
Since the initial data trivially  satisfies
$$\|\omega_0\|_{L^1}\leq \|\omega_L\|_{L^1}+\delta\leq \nu_L+1$$
and
\begin{align*}
 \left\|\omega_0-\omega_{L} \right\|_{   L^2}+\left\|x_2(\omega_0-\omega_{L}) \right\|_{L^1}\leq \delta,
\end{align*}
 we know, by Theorem \ref{thm_lamb}, that
there exists a global weak solution $\omega(t)$ of  \eqref{2d_euler} satisfying
\begin{align}\label{old_est}
\inf_{\tau\in\mathbb{R}}\left\{\left\|\omega(t)-\omega_{L}(\cdot-\tau e_{x_1}) \right\|_{ L^2}+\left\|x_2(\omega(t)-\omega_{L}(\cdot-\tau e_{x_1})) \right\|_{L^1}\right\}\leq   \frac{\varepsilon} {16},\quad \textrm{for all}\ t\geq0.  
\end{align}
 Let $t\geq 0$ be fixed.  By using \eqref{old_est},   we take $\tau=\tau(t)
\in\mathbb{R}$ satisfying
$$\left\|\omega(t)-\omega_{L}(\cdot-\tau e_{x_1}) \right\|_{ L^2}
+\left\|x_2(\omega(t)-\omega_{L}(\cdot-\tau e_{x_1})) \right\|_{L^1}  \leq \frac \varepsilon 8.$$ Then it remains to show
$$\left\|\omega(t)-\omega_{L}(\cdot-\tau e_{x_1}) \right\|_{ L^1}   \leq \frac 7 8 \varepsilon.$$ Indeed, we recall
$\supp\,\omega_L(\cdot-\tau e_{x_1})=\overline{D_+^\tau} $ where
$ {D_+^\tau}:=\{x\in\mathbb{R}^2_+\,|\, |x-\tau e_{x_1}|<1\}$. Thus we compute
  \begin{align*}
& \|\omega(t)-\omega_L(\cdot-\tau e_{x_1})\|_{L^1(\mathbb{R}^2_+)} 
=\|\omega(t)-\omega_L(\cdot-\tau e_{x_1})\|_{L^1(D_+^\tau)}+\|\omega(t) \|_{L^1(\mathbb{R}^2_+\setminus D_+^\tau)}\\ 
&=\|\omega(t)-\omega_L(\cdot-{\tau e_{x_1}})\|_{L^1(D_+^\tau)}+\|\omega(t) \|_{L^1(\mathbb{R}^2_+ )}
- \|\omega(t) \|_{L^1( D_+^\tau)}\\ 
&=\|\omega(t)-\omega_L(\cdot-{\tau e_{x_1}})\|_{L^1(D_+^\tau)}+\|\omega_0 \|_{L^1(\mathbb{R}^2_+ )}
- \|\omega(t) \|_{L^1( D_+^\tau)}\\ 
&\leq \|\omega(t)-\omega_L(\cdot-{\tau e_{x_1}})\|_{L^1(D_+^\tau)}+\|\omega_0-\omega_L \|_{L^1(\mathbb{R}^2_+ )}+\|\omega_L\|_{L^1(\mathbb{R}^2_+ )}
- \|\omega(t) \|_{L^1( D_+^\tau)}\\ 
&\leq \|\omega(t)-\omega_L(\cdot-{\tau e_{x_1}})\|_{L^1(D_+^\tau)}+\delta+\|\omega_L(\cdot-{\tau e_{x_1}})\|_{L^1(D_+^\tau )}
- \|\omega(t) \|_{L^1( D_+^\tau)}\\ 
&\leq 2 \|\omega(t)-\omega_L(\cdot-{\tau e_{x_1}})\|_{L^1(D_+^\tau)}+\frac{\varepsilon}{4} \\ 
&\leq 2 \sqrt{|D_+^\tau|}\cdot \|\omega(t)-\omega_L(\cdot-{\tau e_{x_1}})\|_{L^2(D_+^\tau)}+\frac{\varepsilon}{4} \\ 
&\leq 2 \sqrt{\pi/2}\cdot\frac{\varepsilon} 8+\frac{\varepsilon}{4} \leq \frac {7 }{8}  \varepsilon.
\end{align*} 
Lastly, due to $\delta<\varepsilon/4<\varepsilon$, we can take $\tau(0)=0$. \end{proof}

\subsection{Estimates on the shift $\tau(t)$}
{
We first estimate the amount of ``shell'' in the forward direction of the Lamb dipole.
}
\begin{lemma}\label{lem:lamb_error} 
There exist absolute constants $c_0>0$, $\kappa_0\in(0,1/2)$ such that for any $\kappa \in[0,\kappa_0],$
\begin{equation}\label{claim_lamb_1}
\int_{\{x\in\mathbb{R}^2_+\,|\,x_1\geq 1-\kappa\}}\omega_L(x)\,dx\geq c_0 \kappa^{3}.
\end{equation}

\end{lemma}
\begin{proof}

For $\kappa\in(0,1/2)$, we consider the rectangle $E_{\kappa}$ whose four corners (in Cartesian coordinates) are at 
$$
(1-{\kappa},\frac{\sqrt{\kappa}}4),\quad (1-\frac{3\kappa}4,\frac{\sqrt{\kappa}}4),\quad (1-\frac{3\kappa}4,\frac{\sqrt{\kappa}}2),\quad (1-{\kappa},\frac{\sqrt{\kappa}}2).
$$
Then the rectangle $E_{\kappa}$ has area 
$$\frac{\kappa}4\cdot\frac{\sqrt{\kappa}}{4}=\frac{\kappa^{3/2}}{16}
$$
and is contained in the region $D_+\cap \{ x_1\geq 1-\kappa\}$. Moreover,
it satisfies, for any $x\in E_\kappa$,
$$
1-\kappa \leq r \leq 1-\frac \kappa 4.
$$
 By recalling the definition \eqref{lamb}, we estimate, for $x\in E_\kappa$, 
$$\omega_L(x)=g(r)\sin\theta\geq g(r)x_2
\geq g(r)\cdot\frac{\sqrt\kappa}{4}.
$$  
Due to $g(1)=0$ and $g'(1)<0$, we may use the Taylor expansion near $r=1$ to conclude that there exist constants $\kappa_0\in(0,1/2)$ and $c>0$
such that for any $x\in E_\kappa$,
$$
g(r)\geq c\kappa
$$ whenever $\kappa\in(0,\kappa_0)$.
Thus we obtain
\begin{equation*}
\int_{\{x\in\mathbb{R}^2_+\,|\,x_1\geq 1-\kappa\}}\omega_L(x)\,dx\geq\int_{E_\kappa}\omega_L(x)\,dx\geq 
c\kappa\cdot\frac{\sqrt{\kappa}}{4}\cdot \frac{\kappa^{3/2}}{16}
\geq 
 c_0 \kappa^{3}.
\end{equation*} This gives the lemma. 
\end{proof}
 {We note that any shift $\tau$ satisfying the stability estimate\eqref{lamb_stab} does not have to be continuous. In fact, for each time, the value $\tau(t)$ can be multiply defined (if we ask the shift to satisfy the estimate \eqref{lamb_stab} only). Despite  this observation, 
in the following lemma, we shall prove that 
the position   $\tau(t)$ of shift  
remains close to $\tau(t')$ when $t$ is close to $t'$,
}
  under the additional assumption $\omg_0 \in H^3$. 
  This assumption will be removed later on in Proposition \ref{prop:travel-speed_lamb}. 
\begin{lemma}\label{lem:no_high_jump} 
There exist constants $\varepsilon_1, \tilde{K}>0$ such that 
for any $\omega_0\in L^1\cap L^\infty(\mathbb{R}^2)$ with 
\begin{equation}\label{assum_lem_h3}
\omega_0\in H^3(\mathbb{R}^2),
\end{equation}
if there is   a function $\tau:[0,\infty)\to\mathbb{R}$ 
satisfying 
\begin{align*}
\sup_{t\geq 0}\left\{\left\|\omega(t)-\omega_{L}(\cdot-\tau(t) e_{x_1}) \right\|_{ L^1
}
\right\}\leq \varepsilon 
\end{align*}  for some $\varepsilon\in(0,\varepsilon_1)$, then 
for each $T>0$,  there exists a constant $c_{\omega_0,T}>0$ such that
for any $t,t'\in[0,T]$, 
the function $\tau$  satisfies
 \begin{equation*}
	\begin{split}
		|\tau(t) - \tau(t')  | \leq \tilde{K}\varepsilon^{1/3}
	\end{split}
\end{equation*} whenever $|t-t'|\leq c_{\omega_0,T}\cdot \varepsilon$.
\end{lemma}
\begin{proof}
\begin{enumerate}

\item Denote
$$ {D_+}:=\{x\in\mathbb{R}^2_+\,|\, |x|<1\},$$
$$ {D_+^{\tau}}:=\{x\in\mathbb{R}^2_+\,|\, |x-\tau e_{x_1}|<1\},$$
 for $\tau\in\mathbb{R}$, 
$$\omega_L^{\tau}:=\omega_L(\cdot-\tau e_{x_1})$$
and $$f(\tau):=\|\omega_L-\omega_L^{\tau}\|_{L^1(\mathbb{R}^2_+)}.$$ 
Then
 $f$ is continuous on $\mathbb{R}$,
$$f(0)=0,$$
$$f(\tau)>0,\quad |\tau|>0,$$
$$f(\tau)=c_f:=2\|\omega_L\|_{L^1(\mathbb{R}^2_+)}>0,\quad |\tau|\geq 2.$$
In particular, due to \eqref{claim_lamb_1} of Lemma \ref{lem:lamb_error}, we have, for any $|\tau|\leq \kappa_0$,
\begin{equation}\label{claim_lamb_2}
f(\tau)\geq 2 c_0|\tau|^3,
\end{equation} where $\kappa_0\in(0,1/2)$ is the constant from the lemma.
We set 
$F:[0,\infty)\to\mathbb{R}$ by
$$F(s)=\inf_{|\tau|\geq s}f(\tau).$$ Then $F$ is continuous,   non-decreasing and satisfies
$$F(0)=0, \quad 0<F(s)<c_f\quad \mbox{for}\quad  s>0,
 \quad F(s)=c_f\quad \mbox{for}\quad  s\geq 2.$$
 We take $\varepsilon_1>0$ satisfying
 \begin{equation}\label{F}
 F(\kappa_0)= 4\varepsilon_1.
 \end{equation} 


\item For any $\omega_0\in L^1\cap L^\infty$ with 
$\omega_0\in H^3(\mathbb{R}^2)$,
we denote $h_{\omega_0}:[0,\infty)\to\mathbb{R}$ by
$$
h_{\omega_0}(T):=\sup_{t\in[0,T]}\|\partial_t\omega(t)\|_{L^\infty},\quad T\geq0,
$$ where $\omega$ is the corresponding solution for the initial data $\omega_0$. It is well-defined since   the assumption \eqref{assum_lem_h3} guarantees
$$\sup_{t\in[0,T]}\|\nabla\omega(t)\|_{L^\infty}<\infty$$  for any finite $T>0$.
Let's suppose that there is  a function $\tau:[0,\infty)\to\mathbb{R}$ 
satisfying 
\begin{align*}
\sup_{t\geq 0}\left\{\left\|\omega(t)-\omega_{L}^{\tau(t)} \right\|_{ L^1
}
\right\}\leq \varepsilon 
\end{align*}  for some $\varepsilon\in(0,\varepsilon_1)$.
Then for any fixed $T>0$, we can compute for any $t,t'\in[0,T]$,
  \begin{align*}
  & \|\omega_L^{\tau(t')}-\omega_L^{\tau(t)}\|_1
 =\|\omega_L^{\tau(t')}-\omega_L^{\tau(t)}\|_{L^1(D_+^{\tau(t')}\cup D_+^{\tau(t)})}\\
&\quad \leq  \|\omega_L^{\tau(t')}-\omega(t')\|_{L^1(D_+^{\tau(t')}\cup D_+^{\tau(t)})}+  \|\omega(t')-\omega(t)\|_{L^1(D_+^{\tau(t')}\cup D_+^{\tau(t)})}+  \|\omega(t)-\omega_L^{\tau(t)}\|_{L^1(D_+^{\tau(t')}\cup D_+^{\tau(t)})}\\
 &\quad \leq 2\varepsilon+\|\omega(t')-\omega(t)\|_{L^1(D_+^{\tau(t')}\cup D_+^{\tau(t)})} \leq 2\varepsilon+2|D_+|\cdot \|\omega(t')-\omega(t)\|_{L^\infty}  \\\
 &\quad \leq 2\varepsilon+2|D_+| \cdot h_{\omega_0}(T)\cdot |t'-t|. 
\end{align*} Thus,
by choosing
$$c_{\omega_0,T}:=\left(2|D_+| \cdot h_{\omega_0}(T)\right)^{-1}>0,\quad T>0,$$
we conclude that
 for any $t,t'\in[0,T]$, 
the condition $$|t-t'|\leq c_{\omega_0,T}\cdot\varepsilon$$ implies 
  \begin{align*}
  F(|\tau(t')-\tau(t)|)\leq f(\tau(t')-\tau(t))=\|\omega_L -\omega_L^{\tau(t')-\tau(t)}\|_{L^1}= \|\omega_L^{\tau(t')}-\omega_L^{\tau(t)}\|_{L^1}
  \leq 3\varepsilon\leq 3\varepsilon_1< 4\varepsilon_1.
\end{align*}  In this case, we obtain 
  $$|\tau(t')-\tau(t)|\leq \kappa_0$$ due to
\eqref{F}  
   so that we can use \eqref{claim_lamb_2} to conclude
  $$
  2 c_0|\tau(t')-\tau(t)|^3\leq f(\tau(t')-\tau(t))\leq 3\varepsilon.
  $$ Thus choosing
  $\tilde{K}:= \left(3/(2c_0)\right)^{1/3}$, we are done.

 \end{enumerate}
 \end{proof}

\subsection{The key proposition and the proof of Theorem \ref{thm_lamb_lower}.}

\begin{proposition}[Traveling speed]\label{prop:travel-speed_lamb} 
There exist absolute constants ${\varepsilon_0}, K>0$ such that for each $M>0$, there exists a constant $\alpha_0=\alpha_0(M)>0$ such that for any  $\varepsilon\in(0,{\varepsilon_0})$ and
for any   $\omega_0\in L^1\cap L^\infty(\mathbb{R}^2)$ with $x_2\omega_0\in L^1$,
$$\omega_0\geq 0 \quad\mbox{on}\quad \mathbb{R}^2_+,\quad\omega_0:\,\mbox{odd-symmetric with respect to $x_1$-axis}$$
 and
 $$ \|\omega_0\|_{L^\infty}\leq M,$$ 
 {if
\begin{align}\label{assum_prop}
 \left\|\omega_0-\omega_{L} \right\|_{L^1\cap L^2}+\left\|x_2(\omega_0-\omega_{L}) \right\|_{L^1}\leq \frac{1}{2}\delta,
\end{align} where $\delta=\delta(\varepsilon)>0$ is the constant from Corollary \ref{cor_lamb},  
then any shift function $\tau(t)$ satisfying \eqref{assum_lem_stab} of Corollary \ref{cor_lamb} }
 satisfies
  \begin{equation}\label{eq:travel-est_lamb} 
		\begin{split}
			|\tau(t + \alp ) - \tau(t) - W_L \alp  | \le {K}\varep^{1/3} 
		\end{split}
	\end{equation}   for any $t\geq0$ and for any $\alp \in [0,\alp_0]$.
	 
\end{proposition}

\begin{proof} 

\begin{enumerate}

\item  We denote, for $\tau\in\mathbb{R}$ and $a>0$,
$$ {D_+^{\tau,a}}:=\{x\in\mathbb{R}^2_+\,|\, |x-\tau e_{x_1}|<a\},\quad 
 {D_+^{\tau}}:= {D_+^{\tau,1}}.$$ We recall  that $\omega_L(t,x):=\omega_L(x-tW_Le_{x_1})$ (with some abuse of notation) is a traveling wave solution of \eqref{2d_euler}. The velocity of  Lamb's dipole 
 $$u_L(t, x):=\mathcal{K}[\omega_{L}(t,\cdot_x)](x)=\mathcal{K}[\omega_{L}(\cdot_x-tW_Le_{x_1})](x)$$
  is Lipschitz in space-time $\mathbb{R}^2\times\mathbb{R}_{\geq0}$. Say $C_{Lip}>0$ be the Lipschitz constant. We denote
  the particle trajectory map $\phi_L$ obtained from solving the following ODE system:\\
	$$\frac{d}{dt}\phi_L(t,(t_0,x))=u_L(t,\phi_L(t,(t_0,x)))\quad\mbox{for }t>0\quad\mbox{and } \quad \phi_L(t_0,(t_0,x))=x\in\mathbb{R}^2,$$
	
  Then, we simply observe
$$ \phi_L(t,(t_0,D_+^{W_Lt_0}))= D_+^{W_Lt}, \quad t,t_0\geq 0.
$$

\item Let 
${\varepsilon_0}\in(0,\min(\varepsilon_1,1/2))$ and set $K>0$ large enough so that 
\begin{equation}\label{take_K}
K\geq 4(\tilde{K}+1)\quad\mbox{and}\quad  c_0((1/16)K)^3\geq 3,
\end{equation}
 where $\varepsilon_1, \tilde{K}>0$ are the   constants from Lemma \ref{lem:no_high_jump} and $c_0>0$ is the absolute constant
 from \eqref{claim_lamb_1} from Lemma \ref{lem:lamb_error}. In the sequel, ${\varepsilon_0}>0$ will be chosen small enough.\\
 
Let $\omega_0$ satisfy the assumptions in this proposition   for some $\varepsilon\in(0,{\varepsilon_0})$. Then, by Corollary \ref{cor_lamb},  the corresponding solution satisfies \eqref{assum_lem_stab} for some shift function $\tau(t)$. Let's say $u(t):=\mathcal{K}[\omega(t)]$ be the corresponding velocity.\\

We first prove \eqref{eq:travel-est_lamb} by assuming an extra condition
\begin{equation}\label{assump_arti}
 \omega_0\in H^3(\mathbb{R}^2) 
\end{equation}  
while it will be shown clearly in the proof  that all the resulting constants $\alpha_0, {\varepsilon_0}, K>0$ are not relevant quantitatively to the artificial assumption \eqref{assump_arti}. At the end, we will remove this artificial assumption.
	
\medskip
	
\item Let $T>0$ be arbitrary.  We will carefully trace dependence coming from $T$ (if any) so that the resulting constants $\alpha_0, {\varepsilon_0}, K>0$ are independent of the choice of $T>0$.
Fix some $t_0\in[0,T]$ and $\tau(t_{0})$.  
We shall consider the flow of the half-disk region $D_+^{\tau(t_0)}$ via $u(t,\cdot)$.\ \\

Our proof is done once we show the following statement:	\ \\

	\noindent \textbf{Bootstrap hypotheses}.     We have, for all $t \in [t_0,t_0+\alp_0 ]$,  \begin{equation} \tag{B1'} 
		\begin{split}
			|\tau(t) - (\tau(t_0) + W_L(t-t_0))| \le K \varep^{1/3}
		\end{split}
	\end{equation}  and 
	 \begin{equation}\tag{B2'}
		\begin{split}
\phi(t,(t_0,D_+^{\tau(t_0)}))\subset D_+^{\tau(t_0)+W_L(t-t_0),\, 1+ (1/2)K\varepsilon^{1/3}}, 
		\end{split}
	\end{equation}  where 
	$\phi$ is the particle trajectory map obtained from solving the following ODE system:\\
	$$\frac{d}{dt}\phi(t,(t_0,x))=u(t,\phi(t,(t_0,x)))\quad\mbox{for }t>0\quad\mbox{and } \quad \phi(t_0,(t_0,x))=x\in\mathbb{R}^2,$$
	
	We note that the hypotheses are simply valid at $t = t_0$. \ \\
	
\item	We first prove the following claim:
	
	\medskip
	
	\indent  \textbf{Initial claim}. 
There exists	 a constant  $\eta>0$ depending only on 
$	 \omega_0,T, \varepsilon, W_L$ such that,  for any $t\in[t_0,t_0+\eta]\cap[0,T]$, we have
\begin{equation} \tag{B1-} 
		\begin{split}
			|\tau(t) - (\tau(t_0) + W_L(t-t_0))| \le \frac{1}{2}K \varep^{1/3}
		\end{split}
	\end{equation}  
	and 
	 \begin{equation}\tag{B2-}
		\begin{split}
\phi(t,(t_0,x))\in D_+^{\tau(t_0)+W_L(t-t_0),\, 1+ (1/2)K\varepsilon^{1/3}},\quad
x\in D_+^{\tau(t_0)}.
		\end{split}
	\end{equation} 

	\medskip
	The above claim is essentially done by  Lemma \ref{lem:no_high_jump}. Indeed, the lemma says
	$$	|\tau(t) - \tau(t_0)  | \leq \tilde{K}\varepsilon^{1/3}$$
	whenever $t\in[t_0,t_0+ c_{\omega_0,T}\cdot \varepsilon]\cap [0,T]$. Here, $c_{\omega_0,T}>0$ is the constant from Lemma \ref{lem:no_high_jump}. Thus, on this interval,
	$$	|\tau(t) - (\tau(t_0) + W_L(t-t_0))| \leq \tilde{K}\varepsilon^{1/3} +W_L(t-t_0). 
	$$
We take a constant $\eta>0$ so that
$$\eta\leq c_{\omega_0,T}\cdot \varepsilon\quad\mbox{and}\quad
W_L\eta\leq \tilde{K}\varepsilon^{1/3}. 
$$ This choice of $\eta$ guarantees that   (B1-) holds for any $t\in[t_0,t_0+\eta]\cap [0,T]$ due to 
$$	|\tau(t) - (\tau(t_0) + W_L(t-t_0))| \leq 2\tilde{K}\varepsilon^{1/3} 
\leq \frac 1 2 K\varepsilon^{1/3}.   $$
 
For the claim for (B2-), we note that the flow speed is uniformly bounded: 
\begin{equation}\label{unif_speed}
\|
u(t)
\|_{L^\infty}\leq C\|\omega(t)\|_{L^1}^{1/2}\|\omega(t)\|_{L^\infty}^{1/2}
= C\|\omega_0\|_{L^1}^{1/2}\|\omega_0\|_{L^\infty}^{1/2}=:C_{\omega_0}.
\end{equation}
Thus, for any $x_0\in D_+^{\tau(t),1}$, we have, for any $t\in\mathbb{R}$,
$$
x:=\phi(t,(t_0,x_0))\in D_+^{\tau(t),\, 1+ C_{\omega_0}|t-t'|}.$$ Then, we compute
\begin{equation*}
	\begin{split}
		|x-(\tau(t_0)+W_L(t-t_0))e_{x_1}|&\leq |x-\tau(t_0)e_{x_1}|+W_L|t-t_0|\\
		&\leq  \left(1+ C_{\omega_0}|t-t'|\right) +W_L|t-t_0|\\
		&=  1+ (C_{\omega_0}+W_L)|t-t'|.
	\end{split}
\end{equation*} 
Once we make $\eta>0$ smaller than before 
 (if necessary) in order to have
 \begin{equation}\label{take_eta}
 (C_{\omega_0}+W_L)\eta \leq (7/16)K\varepsilon^{1/3},
 \end{equation}
we obtain   (B2-) for the short interval $[t_0,t_0+\eta]\cap [0,T]$.\\

\item From now on, we may assume that (B1') and (B2') are valid for all $t \in [t_0,t^*]$ with some $t^*>t_0$. The existence of such $t^*$ is guaranteed by Initial claim (B1-) and (B2-).  We shall prove the following claim: \ \\

\indent  \textbf{Bootstrap claim}. For each $M$, there exists a small  constant $\alp_0>0$ depending \textit{only} on $M$ such that if $t^* \leq  t_0 + \alp_0 
$, 
 then actually (B1') and (B2') hold for any $t\in[t_0,t^*]$ with coefficient constants (of $K$)  $1$ and $1/2$ replaced with $1/8$ and $1/16$, respectively. \textit{i.e.} we claim, for $t\in[t_0,t^*]$,
\begin{equation} \tag{B1*} 
		\begin{split}
			|\tau(t) - (\tau(t_0) + W_L(t-t_0))| \le \frac{1}{8}K \varep^{1/3}
		\end{split}
	\end{equation}  and 
	 \begin{equation}\tag{B2*}
		\begin{split}
\phi(t,(t_0,D_+^{\tau(t_0)}))\subset D_+^{\tau(t_0)+W_L(t-t_0),\, 1+ (1/16)K\varepsilon^{1/3}}.
		\end{split}
	\end{equation} 

To verify   Bootstrap claim  (B2*), we fix
any $x_0\in D_+^{\tau(t_0)}$ and
compute with $\phi(t)=\phi(t,(t_0,x_0))$ that 
\begin{equation*}
\begin{split}
	\frac{d}{dt} \phi(t)    = u(t,\phi(t) ) &= u(t,\phi(t) ) - u_L(W_L^{-1}{\tau(t)}, \phi(t)) \\
	&\qquad  + u_L (W_L^{-1}{\tau(t)},\phi(t)) - u_L(W_L^{-1}{\tau(t_0)}+(t-t_0),\phi(t))  \\
	&\qquad + u_L(W_L^{-1}{\tau(t_0)}+(t-t_0),\phi(t))\\
& =:I(t)+II(t)+III(t).
\end{split}
\end{equation*}
From the stability assumption  \eqref{assum_lem_stab}, we have 
\begin{equation*}
	\begin{split}
	|I(t)|&
=|\mathcal{K}[\omega(t)](\phi(t)) -\mathcal{K}[\omega_{L}(\cdot_x-\tau(t)e_{x_1})](\phi(t))|\\ &	
	\leq
\|\mathcal{K}[\omega(t)-\omega_{L}(\cdot_x-\tau(t)e_{x_1})]\|_{L^\infty}\\ &	
	\leq C\|\omega(t)-\omega_{L}(\cdot_x-\tau(t)e_{x_1})\|_{L^1}^{1/2}\|\omega(t)-\omega_{L}(\cdot_x-\tau(t)e_{x_1})\|_{L^\infty}^{1/2}\\ &  \leq C_1(\|\omega_0\|_{L^\infty}+1)^{1/2} \varep^{1/2},\quad t\geq 0,
	\end{split}
\end{equation*} where $C_1>0$ is an absolute constant.
For $II(t)$, we use Lipschitzness (in space-time) of $u_L$ and the hypotheses (B1') on $[t_0,t^*]$ to get,
for any $t\in[t_0,t^*]$,
\begin{equation*}
	\begin{split}
	|II(t)|&\leq
C_{Lip}	\cdot |
	 W_L^{-1}{\tau(t)} -\left(W_L^{-1}{\tau(t_0)}+(t-t_0)\right)|\\
	 &=C_{Lip}W_L^{-1}\cdot |\tau(t) - (\tau(t_0) + W_L(t-t_0))| \le C_{Lip}W_L^{-1}\cdot K \varep^{1/3}.
	\end{split}
\end{equation*}
We denote
$$\psi(t):=\phi_L(W_L^{-1}{\tau(t_0)}+(t-t_0),(W_L^{-1}{\tau(t_0)},x_0)).$$ Then $\psi$ satisfies 
$$\psi(t_0)=x_0\quad\mbox{and}\quad  \psi(t)\in D_+^{\tau(t_0)+W_L(t-t_0),\, 1}\quad\mbox{for any} \quad  t\geq t_0$$ and
\begin{equation*}
	\begin{split}\frac{d}{dt}\psi(t)&=
\frac{d}{dt}\phi_L(W_L^{-1}{\tau(t_0)}+(t-t_0),(W_L^{-1}{\tau(t_0)},x_0))\\&
=u_L(W_L^{-1}{\tau(t_0)}+(t-t_0),\phi_L(W_L^{-1}{\tau(t_0)}+(t-t_0),(W_L^{-1}{\tau(t_0)},x_0)))\\
&=u_L(W_L^{-1}{\tau(t_0)}+(t-t_0),\psi(t)). \end{split}
\end{equation*}
Using the above bounds and comparing the equations for $\phi$ and $\psi$, we see that, for $t\in[t_0,t^*]$, \begin{equation*}
\begin{split}
	\frac{d}{dt}|\phi(t) - \psi(t)| & \le |u_L(W_L^{-1}{\tau(t_0)}+(t-t_0),\phi(t)) -u_L(W_L^{-1}{\tau(t_0)}+(t-t_0),\psi(t)) |  \\
	&\qquad\qquad +C_1(\|\omega_0\|_{L^\infty}+1)^{1/2} \varep^{1/2}+C_{Lip}W_L^{-1}\cdot K \varep^{1/3} \\
	& \le C_{Lip}|\phi(t)-\psi(t)| + C_2\cdot(1+\|\omega_0\|_{L^\infty}) \varep^{1/3}.
\end{split}
\end{equation*} 
With Gronwall's inequality, we deduce for $t\in[t_0,t^*]\subset[t_0,t_0+\alpha_0 ]$ that \begin{equation*}
\begin{split}
	|\phi(t)-\psi(t)|& \le e^{C_{Lip}(t^*-t_0)}\cdot\int_{t_0}^{t^*}    
C_2(1+\|\omega_0\|_{L^\infty})\varep^{1/3} ds. \\
	& \le   \left( e^{ C_{Lip}\alp_0  }C_2(1+M)\alp_0\right)\cdot  \varep^{1/3}.  
\end{split}
\end{equation*} We take $\alp_0>0$ small enough so that
$$
\left(e^{ C_{Lip}\alp_0  }C_2(1+M)  \alp_0\right)\leq \frac 1 {16} K.
$$
Since we know $\psi(t)\in D_+^{\tau(t_0)+W_L(t-t_0),\, 1}$, the above estimate shows
	 \begin{equation*} 
		\begin{split}
\phi(t)\in D_+^{\tau(t_0)+W_L(t-t_0),\, 1+ (1/16)K\varepsilon^{1/3}},
		\end{split}
	\end{equation*} which is Bootstrap claim (B2*) on $[t_0,t^*]$ whenever $t^*\leq t_0+\alpha_0$.\\

	To prove (B1*) on $[t_0,t^*]$ when $t^*\leq t_0+\alpha_0$, we
	denote   $$A_t:=\phi(t,(t_0,D_+^{\tau(t_0)})),\quad  t\geq t_0,
	$$ and decompose$$
	\omega(t,x)=\omega(t,x)1_{A_t}(x)+\omega(t,x)1_{\mathbb{R}^2_+\setminus A_t}(x)=:\Omega^1(t,x)+\Omega^2(t,x).
	$$

As before, we 	denote, for $\tau\in\mathbb{R}$, 
$$\omega_L^{\tau}:=\omega_L(\cdot-\tau e_{x_1}).$$ Then we note $$\|\Omega^2(t)\|_{L^1}=\|\Omega^2(t_0)\|_{L^1}\leq
	\|\omega(t_0)-\omega_L^{\tau(t_0)}\|_{L^1({\mathbb{R}^2_+\setminus A_{t_0}})}\leq
	\|\omega(t_0)-\omega_L^{\tau(t_0)}\|_{L^1}
	\leq  \varepsilon$$ since 	 $\omega_L^{\tau(t_0)}$ is supported in $A_{t_0}=D_+^{\tau(t_0)}$. \\
	
For	a contradiction, let's assume that (B1*) on $[t,t^*]$ with   $t^*\leq t_0+\alpha_0$, fails. \textit{i.e.}  there is some $t'\in[t,t^*]$ satisfying
\begin{equation*}  
		\begin{split}
			|\tau(t') - (\tau(t_0) + W_L(t'-t_0))| > \frac{1}{8}K \varep^{1/3}.
		\end{split}
	\end{equation*} We may assume $$\tau(t') - (\tau(t_0) + W_L(t'-t_0)) > \frac{1}{8}K \varep^{1/3}$$ since the other case can be considered similarly.
Because we already obtained   (B2*) on $[t_0,t^*]$ with   $t^*\leq t_0+\alpha_0$,	 we observe	 
$$A_{t'}\subset  D_+^{\tau(t_0)+W_L(t'-t_0),\, 1+ (1/16)K\varepsilon^{1/3}}\subset 
\{x\in\mathbb{R}^2\,|\, x_1\leq 
\left(\tau(t_0)+W_L(t'-t_0)\right)+\left( 1+ (1/16)K\varepsilon^{1/3}\right)
\}.
$$ 
We make ${\varepsilon_0}>0$ smaller than before (if necessary) to have
\begin{equation}\label{take_e_2}
 (1/16)K{\varepsilon_0}^{1/3}\leq \kappa_0, 
\end{equation}
where $\kappa_0>0$ is the constant in \eqref{claim_lamb_1} from Lemma \ref{lem:lamb_error}.
Thus we can compute	
			 \begin{equation*} 
		\begin{split}
		\|\omega(t')-\omega_L^{\tau(t')}\|_{L^1}&\geq 
		\|\Omega^1(t')-\omega_L^{\tau(t')}\|_{L^1}-\|\Omega^2(t')\|_{L^1}\\
		&\geq \|\Omega^1(t')-\omega_L^{\tau(t')}\|_{L^1(\mathbb{R}^2_+\setminus A_{t'})}-\varepsilon= \| \omega_L^{\tau(t')}\|_{L^1(\mathbb{R}^2_+\setminus A_{t'})}-\varepsilon\\
		&\geq 
		\| \omega_L^{\tau(t')}\|_{L^1(\{x_1>
		\left(\tau(t_0)+W_L(t'-t_0)\right)+\left( 1+ (1/16)K\varepsilon^{1/3}\right)
		\})}-\varepsilon\\
			&\geq 
		\| \omega_L^{\tau(t')}\|_{L^1(\{x_1>
		 \tau(t')+\left( 1- (1/16)K\varepsilon^{1/3}\right)
		\})}-\varepsilon\\
		&=
		\| \omega_L\|_{L^1(\{x_1\geq  1-(1/16)K\varepsilon^{1/3}
		\})}-\varepsilon\\&\geq c_0((1/16)K)^3\varepsilon-\varepsilon=
		\left(c_0((1/16)K)^3-1\right)\varepsilon,
 		\end{split}
	\end{equation*} where the last inequality  and the constant $c_0$ are from   \eqref{claim_lamb_1} of Lemma \ref{lem:lamb_error} (due to \eqref{take_e_2}). 
	Thanks to the assumption \eqref{take_K} on $K$, we conclude
		 \begin{equation*} 
		\begin{split}
		\|\omega(t')-\omega_L^{\tau(t')}\|_{L^1}&\geq 2\varepsilon,
 		\end{split}
	\end{equation*} which is a contradiction to \eqref{assum_lem_stab}. Hence we obtain (B1*) on $[t_0,t^*]$ when   $t^*\leq t_0+\alpha_0$.\\
	
\item	Lastly, we use the  continuity argument to finish the proof since we can extend the valid interval $[t_0,t^*]$
for (B1') and (B2') (so (B1*) and (B2*), too)
 as longer as we want since 
		 we can pour Initial claim (B1-)  	into $[t^*, t^*+\eta]\cap[0,T]$. More precisely,  Initial claim (B1-) implies for any $t\in[t^*,t^*+\eta]\cap[0,T]$,
	\begin{equation*}  
		\begin{split}
			|\tau(t) - (\tau(t^*) + W_L(t-t^*))| \le \frac 1 2 K \varep^{1/3}.
		\end{split}
	\end{equation*}  
We add the above estimate into (B1*) (for $t=t^*$) in order to obtain, for $t\in[t^*,t^*+\eta]\cap[0,T]$,
$$ 
|\tau(t) - (\tau(t_0) + W_L(t-t_0))| \le \frac{1}{8}K \varep^{1/3}+\frac 1 2 K \varep^{1/3}\leq  K \varep^{1/3}.
$$ Thus we have (B1') on the extended interval $[t_0,t^*+\eta]\cap[0,T]$. \\

To extend (B2') up to $[t_0,t^*+\eta]\cap[0,T]$, we recall (B2*) for $t=t^*$:
$$\phi(t^*,(t_0,D_+^{\tau(t_0)}))\subset D_+^{\tau(t_0)+W_L(t^*-t_0),\, 1+ (1/16)K\varepsilon^{1/3}}.$$ Thanks to the uniform bound \eqref{unif_speed} of the flow speed, we have, 
for  $t\in[t^*, t^*+\eta]$,
  	 \begin{equation} \label{comp_short_eat}
		\begin{split}
\phi(t,(t_0,D_+^{\tau(t_0)}))
&\subset \phi(t,(t^*,
 D_+^{\tau(t_0)+W_L(t^*-t_0),\, 1+ (1/16)K\varepsilon^{1/3}})\\
 &\subset   D_+^{\tau(t_0)+W_L(t^*-t_0),\, 1+ (1/16)K\varepsilon^{1/3}+C_{\omega_0}\eta}
		\end{split}
	\end{equation}
	
	We claim   
	$$D_+^{\tau(t_0)+W_L(t^*-t_0),\, 1+ (1/16)K\varepsilon^{1/3}+C_{\omega_0}\eta}\subset D_+^{\tau(t_0)+W_L(t-t_0),\, 1+ (1/2)K\varepsilon^{1/3}}.$$ Indeed, for any $y\in D_+^{\tau(t_0)+W_L(t^*-t_0),\, 1+ (1/16)K\varepsilon^{1/3}+C_{\omega_0}\eta}$, 	we compute
\begin{equation*} 
		\begin{split}
 |y-\left(\tau(t_0)+W_L(t-t_0)\right)e_{x_1}|&\leq
 |y-\left(\tau(t_0)+W_L(t^*-t_0)\right)e_{x_1}|+W_L|t^*-t|\\ 
  &\leq  1+ (1/16)K\varepsilon^{1/3}+C_{\omega_0}\eta+W_L\eta.
		\end{split}
	\end{equation*}	
		Since we assumed the condition \eqref{take_eta}, we obtained the above claim.\\
		
 Together with
\eqref{comp_short_eat}, it implies,
	  for $t\in[t^*, t^*+\eta]$,
	$$\phi(t,(t_0,D_+^{\tau(t_0)}))\subset  D_+^{\tau(t_0)+W_L(t-t_0),\, 1+ (1/2)K\varepsilon^{1/3}},$$ which gives (B2') on the extended interval $[t_0,t^*+\eta]\cap[0,T]$.\\
 
	 Next, we obtain  Bootstrap claim (B1*) and (B2*) on the extended interval $[t_0,t^*+\eta]\cap[0,T]$. Then we simply repeat the same process above to get that (B1*) and (B2*) are valid on $[t_0,t^*+k\eta]\cap[0,T]$ for any $k\in\mathbb{N}$ until (B1') and (B2') on the \textit{full} interval $[t_0,t_0+\alpha_0]\cap [0,T]$ can be covered. Since we chose $T>0$ arbitrary while our chosen constants   $\alp_0, {\varepsilon_0}, K$ are not relevant to the choice of $T$,  we have proved both (B1') and (B2') for \textit{all} $t_0\geq 0$ whenever $t\in[t_0,t_0+\alpha_0]$.
	
	\item It remains to remove the artificial assumption $\omega_0\in H^3(\mathbb{R}^2)$ in \eqref{assump_arti}. It is not a difficult task since we have chosen   the   constants $\alpha_0, {\varepsilon_0}, K>0$ so that they are not quantitatively related to the assumption. We present a proof below for completeness. \ \\
	
	For a  general data $\omega_0$ under the hypotheses of this proposition with the corresponding solution $\omega(t)$, we simply  take a sequence $\omega_{n,0}\in H^3\cap L^1(\mathbb{R}^2)$ by 
	$$\omega_{n,0}:=\Phi_{1/n}*\omega_0,$$
	where $\Phi\in C_c^\infty(\mathbb{R}^2)$ is a radial, non-negative function with $\int \Phi dx=1$ and $\Phi_a(x):=(1/a^2) \Phi(x/a)$ for $a>0$. In addition, we assume 
$\Phi$ is non-increasing in the radial direction.	
Then, we note that
 $$x_2\omega_{n,0}\in L^1,\quad \omega_{n,0}\geq 0 \quad\mbox{on}\quad \mathbb{R}^2_+,\quad\omega_{n,0}:\,\mbox{odd-symmetric with respect to $x_1$-axis.}$$
In particular, we have  
	 	\begin{equation}\label{artificial}
		 \|\omega_{n,0}\|_{L^\infty}\leq \|\omega_{0}\|_{L^\infty}\leq M. 	
	\end{equation}
	The assumption \eqref{assum_prop} implies that for sufficiently large $n$,
	\begin{align}\label{mid_assum_prop}
 \left\|\omega_{n,0}-\omega_{L} \right\|_{L^1\cap L^2}+\left\|x_2(\omega_{n,0}-\omega_{L}) \right\|_{L^1}\leq  \delta.
\end{align} 
Then we  consider  a sequence of the corresponding solutions $\omega_n(t)$. By a  classical argument (\textit{e.g.} see Proposition 8.2 (ii) in \cite{MB}), 
$\omega_n(t)$ converges in $L^1$ to the unique weak solution $\omega(t)$ for each $t$;
	 \begin{equation}\label{convergence}
	 \lim_{n\to\infty}\|\omega_n(t)-\omega(t)\|_{L^1}=0,\quad t\geq 0.
	 \end{equation}
For each large $n$ satisfying \eqref{mid_assum_prop}, we take a shift function  $\tau_n(\cdot)$ satisfying the stability estimate
\begin{align*}
\sup_{t\geq 0}\left\{\left\|\omega_n(t)-\omega_{L}(\cdot-\tau_n(t) e_{x_1}) \right\|_{ L^1\cap L^2}+\left\|x_2(\omega_n(t)-\omega_{L}(\cdot-\tau_n(t) e_{x_1})) \right\|_{L^1}\right\}\leq \varepsilon. 
\end{align*} 
	 by applying Corollary \ref{cor_lamb} into $\omega_n$. 
	 \item  	 
	 Due to $\omega_{n,0}\in H^3(\mathbb{R}^2)$ and 
\eqref{artificial},	
	  we obtain \eqref{eq:travel-est_lamb} \textit{uniformly} for each large $n$:
$$
			|\tau_n(t + \alp ) - \tau_n(t) - W_L \alp  | \le {K}\varep^{1/3} 
$$  for any $t\geq 0$  and for any $\alpha\in[0,\alpha_0]$. Let $\tau(\cdot)$ be a shift function for the solution $\omega$ satisfying \eqref{assum_lem_stab}. 
As in the proof of Lemma \ref{lem:no_high_jump},
 for each $t\geq 0$, 
 we compute, for each $t\geq0$ and for sufficiently large $n$,
 $$2c_0|\tau_n(t)-\tau(t)|^3\leq \|\omega_L^{\tau_n(t)}-\omega_L^{\tau(t)}\|_{L^1}\leq
 2\varepsilon+\|\omega_n(t)-\omega(t)\|_{L^1}\leq 3\varepsilon,
$$
where the first inequality follows from Lemma \ref{lem:lamb_error}.
 Thanks to \eqref{convergence}, we have the same estimate 
$$
			|\tau(t + \alp ) - \tau(t) - W_L \alp  | \le {K}\varep^{1/3} 
$$ by redefining $K$ sufficiently larger than before. \qedhere \end{enumerate} 
\end{proof}

We are now ready to complete the proof of Theorem  \ref{thm_lamb_lower}.
\begin{proof}[Proof of Theorem \ref{thm_lamb_lower}]

The first part (I) is just  the orbital stability (Corollary \ref{cor_lamb}). \\

The second part (II) is a direct consequence of Proposition \ref{prop:travel-speed_lamb}. Indeed, 
 a simple summation  
of the estimate \eqref{eq:travel-est_lamb} 
  gives \begin{equation*}
	\begin{split}
		|\tau(t) - \tau(0) -W_Lt| \le (K/\alpha_0) (t+\alpha_0)\varepsilon^{1/3}, \qquad t\ge 0. 
	\end{split}
\end{equation*}  Lastly, from Corollary \ref{cor_lamb}, we can take $\tau(0)=0$.
\end{proof}

\subsection{Proof of Theorem  \ref{thm_fil_lamb}}

\begin{proposition}[Trajectories for the perturbation] \label{prop:traj}
{There exists $\varepsilon_1>0$ satisfying the following property:}\\
	Let $\omg_{0}$ satisfy the hypotheses in (I) and (II) from Theorem \ref{thm_lamb_lower}  {for some $\varepsilon<\varepsilon_1$.} Let $\omg(t,x)$ be the unique global solution, with associated flow map $\phi(t,x)$. Then, for any $z = (z_1,z_2)$ with $z_1 < -2$, we have that \begin{equation*}
		\begin{split}
			\tau(t) - \phi^1(t,z) > \frac{W_{L}}{3}t,\qquad t\ge0. 
		\end{split}
	\end{equation*}
\end{proposition}

\begin{proof}  Let $\varepsilon_1>0$ satisfy
$\varepsilon_1<\varepsilon_0$, where $\varepsilon_0$ is the constant from Theorem \ref{thm_lamb_lower}.
We denote by $u$ the velocity field associated with the solution $\omg$. We shall track the trajectory of the point $z$ given in the statement. To this end, we observe that for $y = (y_1,y_2)$ satisfying $y_2\ge 0$ and $y_1 < \tau(t) - \frac{3}{2}$ for some $t\ge0$,\begin{equation}\label{eq:u1-bound}
		\begin{split}
			u^1(t,y) &\le |u^1(t,y)-u^1_{L}(t,y-\tau(t)e_{x_1})| + u^{1}_{L}(t,y-\tau(t)e_{x_1}) \\
			& \le C_{M}\varep^{\frac12} + \frac{W_{L}}{|y_1-\tau(t)|^{2}}
		\end{split}
	\end{equation} We have used \eqref{eq:vel1} (in the case $r>1$)
and the interpolation inequality \begin{equation*}
	\begin{split}
		\nrm{u(t,\cdot) - u_{L}(\cdot - \tau(t)e_{x_1})}_{L^{\infty}} \le C\nrm{\omg(t,\cdot)- \omg(\cdot-\tau(t)e_{x_1})}_{L^1}^{1/2}\nrm{ \omg(t,\cdot)- \omg(\cdot-\tau(t)e_{x_1}) }_{L^{\infty}}^{1/2} \le C_M\varep^{\frac12}. 
	\end{split}
\end{equation*} Now, the proposition readily follows from verifying the following bootstrap hypothesis for all $t\ge0$, with $z$ given in the statement. 
	
	\medskip
	
	\noindent \textbf{Bootstrap Assumption}: $\phi^{1}(t,z) - \frac{W_L}{2}t <  { - \frac{7}{4}}$.
	
	\medskip
	
	\noindent This assumption is satisfied at least for some small time interval containing $t = 0$, since $\phi(t,z)$ is continuous in time. We assume that the above holds on $[0,t^*]$ for some $t^*>0$.
	From \eqref{lamb_main_est}, we have \begin{equation*}
		\begin{split}
			\tau(t) \ge W_{L}t - C_M(1+t)\varep^{\frac13}, \quad t \ge 0. 
		\end{split}
	\end{equation*} Using this together with the bootstrap hypothesis, by taking $\varep_1>0$ smaller if necessary, we have that \begin{equation*}
	\begin{split}
		\phi^1(t,z) < \tau(t) - \frac{3}{2}, \quad t \in [0,t^*]
	\end{split}
\end{equation*} and therefore we may use \eqref{eq:u1-bound} on the same time interval; this gives \begin{equation*}
\begin{split}
	\frac{d}{dt} \phi^1(t,z) \le  C_{M}\varep^{\frac12} + \frac{W_{L}}{|\phi^1(t,z)-\tau(t)|^{2}}
\end{split}
\end{equation*} and hence \begin{equation*}
\begin{split}
	\frac{d}{dt} (\phi^1(t,z) - \frac{W_L}{2}t ) \le  C_{M}\varep^{\frac12} + W_{L}\left( \frac{1}{|\phi^1(t,z)-\tau(t)|^{2}} - \frac12\right)  < 0, 
\end{split}
\end{equation*} by taking $\varep_1>0$ smaller if necessary. Integrating in time and recalling that $z_1 < -2$, we conclude that actually on $t \in [0,t^*]$ we have \begin{equation*}
\begin{split}
	\phi^1(t,z) - \frac{W_L}{2} t < -2. 
\end{split}
\end{equation*} This shows that the bootstrap hypothesis is valid for a slightly longer time interval, which proves the hypothesis for all $t\ge 0$. Finally, using this with \eqref{lamb_main_est} again, we deduce that \begin{equation*}
\begin{split}
	\tau(t) - \phi^1(t,z) - \frac{1}{3}W_{L}t > 0
\end{split}
\end{equation*} by taking $\varep_1>0$ smaller if necessary. 
\end{proof}

We are in a position to conclude Theorem \ref{thm_fil_lamb}. 

\begin{proof}[Proof of Theorem \ref{thm_fil_lamb}]
	We consider a sufficiently small $ \varep>0 $ which satisfies $\varep<\varep_1$ where $\varep_1$ is the constant from Proposition \ref{prop:traj}. Furthermore, for $m := \nrm{\omg_L}_{L^{\infty}}$, we consider the region \begin{equation*}
		\begin{split}
			A_{0} = \{ x \in \supp(\omg_L) : \frac{3m}{4} \le \omg_L(x) \}
		\end{split}
	\end{equation*} and assume that $$\varep < \frac{m}{10} \mathrm{area}(A_{0}).$$ Then, for $\dlt=\dlt(\varep)>0$ from Theorem \ref{thm_lamb_lower}, we can take $\omg_{0} \in C^{\infty}_{c}(\bbR^{2})$ with $\nrm{\omg_{0}}_{L^{\infty}} \le m$ satisfying all the assumptions from Theorem \ref{thm_lamb_lower}. In addition, we may require that there is a connected and  { open} set $A_{1} \subset \bbR^2_+$ satisfying the following properties: \begin{enumerate}
	\item[(i)] $A_{1} \supset A_{0}$ and $\mathrm{area}(A_{1})< 10$; 
	\item[(ii)] $A_{1}  \cap \{ x : \omg_{0}(x)  {>} \frac{3m}{4} \}$ is connected;
	\item[(iii)] $\omg_{0}(A_{1}) \ge \frac{m}{2}$ and for $x \in A_{1}$, $\omg_{0}(x) = \frac{m}{2}$ if and only if $x \in \partial A_{1}$;
	\item[(iv)] there exists a point $z = (z_1,z_2) \in  (A_{1})$ with $ z_1<-2$ and $\omg_{0}(z)  {>} \frac{3m}{4}$. 
\end{enumerate}
Let $\omg(t,x)$ be the unique global solution to 2D Euler with initial data $\omg_0$. We set $\phi$ to be the associated flow map. We now consider the image of the regions $A_{1}$ and $A_{0}$ by $\phi$. We claim the following properties for all $t\ge 0$: \begin{itemize}
	\item For $X(t):= \phi(t,A_{1}) \cap \{ x : \omg(t,x)  {>} \frac{3m}{4} \}$, we have   $$\mathrm{diam}_{x_1}(X(t)) := \sup_{x, x' \in X(t) }|x_1 - x'_1| \ge  {\frac{W_L}{3}t-1}.$$
	\item there exists a point $y$ (depending on $t$) such that $y \in \phi(t,A_{0}) \cap \supp( \omg_{L}(\cdot - \tau(t)e_{x_1}) )$. 
\end{itemize}  Indeed, assume towards a contradiction that $\emptyset = \phi(t,A_{0}) \cap \supp( \omg_{L}(\cdot - \tau(t)e_{x_1}) )$ for some $t>0$. But then, since $\omg(t,\phi(t,A_{0})) \ge  {\frac{m}{2}}$, we have that \begin{equation*}
\begin{split}
	\nrm{\omg(t) - \omg_{L}(\cdot - \tau(t)e_{x_1}) }_{L^{1}} \ge  {\frac{m}{2}} |\mathrm{area}(A_0)| > \varep, 
\end{split}
\end{equation*} which contradicts \eqref{lamb_stab}. 
Then, the first item follows from applying the second item together with Proposition \ref{prop:traj} and the assumption (iv) in the above. \\

Let us now conclude growth of the gradient. Since the flow is area preserving, we have from the assumption (i) that \begin{equation*}
	\begin{split}
		\mathrm{area}(X(t)) < \mathrm{area}(\phi(t,A_{1})) = \mathrm{area}(A_{1}) < 10. 
	\end{split}
\end{equation*}  We consider the set $I(t)$ defined by \begin{equation*}
	\begin{split}
		I(t):= \{ x_1 : X(t) \mbox{ intersects } (\{x_1 \} \times   \bbR_+) \}.
	\end{split}
\end{equation*}  By Fubini's theorem and $\mathrm{diam}_{x_1}(X(t)) \ge  {\frac{W_L}{4}t}$ for any  {$t>t_0:=12/W_L$,} there exists $\bar{x}_1 \in I(t)$ such that \begin{equation}\label{est_l1me}
\begin{split}
|\phi(t,A_{1}) \cap  (\{\bar{x}_1 \} \times   \bbR_+)| \le  { \frac{40}{W_L}\cdot\frac{1}{t}. }
\end{split}
\end{equation} Here, $|\cdot|$ denotes the one-dimensional Lebesgue measure. We have used that $X(t)$ is connected, which follows from the assumption (ii). However, since $\bar{x}_1 \in X(t)$, there exists $y \in  \phi(t,A_{1}) \cap  (\{\bar{x}_1 \} \times   \bbR_+)$ such that $\omg(t,y) \ge \frac{3}{4}m$. Moreover, since $\omg(t,\phi(t,\partial A_1)) = \frac{m}{2}$ (assumption (iii)), there exists $y' \in  \phi(t,\partial A_{1}) \cap  (\{\bar{x}_1 \} \times   \bbR_+)$ such that $\omg(t,y') = \frac{m}{2}$. 
 {From \eqref{est_l1me}, we may assume
that $y'$ satisfies 
$$|y_2-y_2'|\leq \frac{40}{W_L}\cdot\frac{1}{t}.$$}
Finally, by the mean value theorem,  we conclude that \begin{equation*}
\begin{split}
	\sup_{z \in [y,y']}|\rd_{x_2}\omg(t, z)| \ge \frac{1}{|y-y'|} \frac{m}{4} \ge \frac{m W_L}{160}t , 
\end{split}
\end{equation*} where $[y,y']$ is the line segment connecting $y$ and $y'$. This gives $\nrm{\nb\omg(t,\cdot)}_{L^{\infty}} \gtrsim t$. Similarly, $\nrm{\omg(t,\cdot)}_{C^{\alp}} \gtrsim t^\alp$ follows. To obtain the  {${L^p}$--growth of $\nb\omg(t,\cdot)$ for $p\geq1$,}  it suffices to observe that  
\begin{equation*}
\begin{split}
	\left| \left\{ x_1 \in I(t) : |\phi(t,A_{1}) \cap  (\{x_1 \} \times   \bbR_+)| \le  {\frac{80}{W_Lt}} \right\} \right| \ge  {\frac{W_Lt}{8}.} 
\end{split}
\end{equation*} This finishes the proof. \end{proof}

\subsection{Active scalar equations}\label{subsec:gSQG}

In this section, we illustrate the main steps in the proof of Proposition \ref{prop:sqg}, which is largely parallel to that of Theorem \ref{thm_fil_lamb}, using the recent existence and stability theorems of Cao--Qin--Zhan--Zou \cite{CQZZ-stab}.

\begin{itemize}
	\item \textbf{Existence of traveling waves.} The authors in \cite{CQZZ-stab} prove existence of a smooth traveling wave solution $\tht_L^{(\alp)}$ to \eqref{eq:alp-SQG} for $0<\alp<1$ in $\bbR^2$ which is non-negative on $\bbR^2_+$, odd symmetric with respect to the $x_1$--axis, and supported precisely in the unit disc; see \cite[Theorem 1.1]{CQZZ-stab}. Similarly as in \cite{AC2019}, the existence statement is obtained by setting up a maximization problem involving conserved quantities for \eqref{eq:alp-SQG}. 
	\item \textbf{Uniqueness and stability.} In \cite[Theorem 1.3]{CQZZ-stab}, the authors demonstrate the orbital stability of $\tht_L^{(\alp)}$ for $0<\alp<1$. {In particular, their stability statement for $1/2<\alpha<1$} is parallel to Theorem \ref{thm_lamb} obtained in \cite{AC2019} for 2D Euler (gives the stability in the same norms ($L^2$+impulse); however, in the $\alp$-SQG case, stability is conditional in the sense that one needs to assume existence of a \textit{conservative} weak solution corresponding to the perturbed initial data. See \cite{CQZZ-stab} for the precise statement. {As in Corollary \ref{cor_lamb}, the stability norm can be replaced with 
	($L^1$+$L^2$+impulse) without any difficulty.} When the perturbation is sufficiently smooth, then this assumption is guaranteed as long as the solution remains smooth, which is all that we need for Proposition \ref{prop:sqg}. 
	\item \textbf{Quantitative stability.} Based on the stability statement, one can prove an estimate on the traveling speed for the perturbation, which is similar to Proposition \ref{prop:travel-speed_lamb} with the same order of the error $O(\varep^{1/3})$ in the right hand side of \eqref{eq:travel-est_lamb}. (While it does not seem to be stated in \cite{CQZZ-stab}, one can prove that the profile $\tht_L^{(\alp)}$ is given in the form $g^{(\alp)}(r)\sin(\tht)$ in polar coordinates with some function $g^{(\alp)}$ satisfying $(g^{(\alp)})'(1)<0$.) In this proof, we would like to control the difference $\nrm{ u_L^{(\alp)} - u }_{L^\infty}$ under the assumptions $\nrm{ \tht_L^{(\alp)} - \tht }_{L^1\cap L^2} \ll 1$ and $\nrm{ \tht}_{L^{\infty}} \lesssim 1$; this is where the restriction $1/2 < \alp < 1$ comes in. Here, $u_L^{(\alp)}$ and $u$ are the velocities for the $\alp$-SQG equation corresponding to $\tht_L^{(\alp)}$ and $\tht$, respectively.  
	\item \textbf{Perturbed trajectories and instability.} Upon having the estimate for the traveling speed of the perturbation, one can obtain bounds for the particle trajectories for the perturbed initial data and conclude instability by the exact same argument used in Section 2.5. Here, the key is again that we have $\nrm{ u_L^{(\alp)} - u }_{L^\infty} \ll 1$. 
\end{itemize}

\section*{Acknowledgement}

\noindent 
KC has been supported by the National Research Foundation of Korea (NRF-2018R1D1A1B07043065) and by the UBSI Research Fund(1.219114.01) of UNIST.
 IJ has been supported by the Samsung Science and Technology Foundation under Project Number SSTF-BA2002-04. We thank T. Drivas and T. Elgindi for helpful discussions. 

\ \\ 
\bibliographystyle{abbrv}

\begin{thebibliography}{10}
	
	\bibitem{AC2019}
	Ken Abe and Kyudong Choi.
	\newblock Stability of {L}amb dipoles.
	\newblock {\em preprint, arXiv:1911.01795}.
	
	\bibitem{BiCho}
	P~Billant and J-M. Chomaz.
	\newblock Experimental evidence for a new instability of a vertical
	columnarvortex pair in a strongly stratified fluid.
	\newblock {\em J. Fluid Mech.}, 418:167--88, 2000.
	
	\bibitem{CQZZ-stab}
	Daomin Cao, Guolin Qin, Weicheng Zhan, and Changjun Zou.
	\newblock Existence and stability of smooth traveling circular pairs for the
	generalized surface quasi-geostrophic equation.
	\newblock {\em arXiv:2103.04041}.
	
	\bibitem{CCCGW}
	Dongho Chae, Peter Constantin, Diego C\'{o}rdoba, Francisco Gancedo, and
	Jiahong Wu.
	\newblock Generalized surface quasi-geostrophic equations with singular
	velocities.
	\newblock {\em Comm. Pure Appl. Math.}, 65(8):1037--1066, 2012.
	
	\bibitem{Chap1903}
	S.~A. Chaplygin.
	\newblock One case of vortex motion in fluid.
	\newblock {\em Trudy Otd. Fiz. Nauk Imper. Mosk. Obshch. Lyub. Estest.},
	11(11--14), 1903.
	
	\bibitem{Chap07}
	S.~A. Chaplygin.
	\newblock One case of vortex motion in fluid.
	\newblock {\em Regul. Chaotic Dyn.}, 12:219--232, 2007.
	
	\bibitem{Choi2020}
	Kyudong Choi.
	\newblock Stability of {H}ill's spherical vortex.
	\newblock {\em preprint, arXiv:2011.06808}.
	
	\bibitem{Choi2019}
	Kyudong Choi.
	\newblock On the estimate of distance traveled by a particle in a disk-like
	vortex patch.
	\newblock {\em Appl. Math. Lett.}, 97:67--72, 2019.
	
	\bibitem{CJ_Hill}
	Kyudong Choi and In-Jee Jeong.
	\newblock Filamentation near {H}ill's vortex.
	\newblock {\em preprint, arXiv:2107.06035}.
	
	\bibitem{CJ_windin}
	Kyudong Choi and In-Jee Jeong.
	\newblock On the winding number for particle trajectories in a disk-like vortex
	patch of the {E}uler equations.
	\newblock {\em preprint, arXiv:2008.05085}.
	
	\bibitem{CJ}
	Kyudong Choi and In-Jee Jeong.
	\newblock Growth of perimeter for vortex patches in a bulk.
	\newblock {\em Appl. Math. Lett.}, 113:106857, 9, 2021.
	
	\bibitem{CMT1}
	Peter Constantin, Andrew~J. Majda, and Esteban~G. Tabak.
	\newblock Singular front formation in a model for quasigeostrophic flow.
	\newblock {\em Phys. Fluids}, 6(1):9--11, 1994.
	
	\bibitem{CoBa}
	Y~Couder and C~Basdevant.
	\newblock Experimental and numerical study of vortex couples intwo-dimensional
	flows.
	\newblock {\em J. Fluid Mech.}, 173:225--51, 1986.
	
	\bibitem{Den}
	Sergey~A. Denisov.
	\newblock Infinite superlinear growth of the gradient for the two-dimensional
	{E}uler equation.
	\newblock {\em Discrete Contin. Dyn. Syst.}, 23(3):755--764, 2009.
	
	\bibitem{Denisov-merging}
	Sergey~A. Denisov.
	\newblock The centrally symmetric {$V$}-states for active scalar equations.
	{T}wo-dimensional {E}uler with cut-off.
	\newblock {\em Comm. Math. Phys.}, 337(2):955--1009, 2015.
	
	\bibitem{Den2}
	Sergey~A. Denisov.
	\newblock Double exponential growth of the vorticity gradient for the
	two-dimensional {E}uler equation.
	\newblock {\em Proc. Amer. Math. Soc.}, 143(3):1199--1210, 2015.
	
	\bibitem{Do}
	Tam Do.
	\newblock On vorticity gradient growth for the axisymmetric 3{D} {E}uler
	equations without swirl.
	\newblock {\em Arch. Ration. Mech. Anal.}, 234(1):181--209, 2019.
	
	\bibitem{DE}
	Theodore D. Drivas and Tarek M. Elgindi. \textit{Work in preparation.}
	
	\bibitem{EJSVP1}
	Tarek~M. Elgindi and In-Jee Jeong.
	\newblock On singular vortex patches, {I}: Well-posedness issues.
	\newblock {\em Memoirs of the AMS, to appear, arXiv:1903.00833}.
	
	\bibitem{EJSVP2}
	Tarek~M. Elgindi and In-Jee Jeong.
	\newblock On singular vortex patches, {II}: long-time dynamics.
	\newblock {\em Trans. Amer. Math. Soc.}, 373(9):6757--6775, 2020.
	
	\bibitem{FvH}
	J~B Fl\'{o}r and G~J~F van Heijst.
	\newblock An experimental study of dipolar vortex structures in a
	stratifiedfluid.
	\newblock {\em J. Fluid Mech.}, 279:101--33, 1994.
	
	\bibitem{GaPa}
	Francisco Gancedo and Neel Patel.
	\newblock On the local existence and blow-up for generalized {SQG} patches.
	\newblock {\em Ann. PDE}, 7(1):Paper No. 4, 63, 2021.
	
	\bibitem{HeKi}
	Siming He and Alexander Kiselev.
	\newblock Small-scale creation for solutions of the {SQG} equation.
	\newblock {\em Duke Math. J.}, 170(5):1027--1041, 2021.
	
	\bibitem{Niel1}
	J.~S. Hesthaven, J.~P. Lynov, A.~H. Nielsen, J.~Juul Rasmussen, M.~R. Schmidt,
	E.~G. Shapiro, and S.~K. Turitsyn.
	\newblock Dynamics of a nonlinear dipole vortex.
	\newblock {\em Phys. Fluids}, 7(9):2220--2229, 1995.
	
	\bibitem{Iftimie_lec}
	Drago\c{s} Iftimie.
	\newblock Large time behavior in perfect incompressible flows.
	\newblock In {\em Partial differential equations and applications}, volume~15
	of {\em S\'{e}min. Congr.}, pages 119--179. Soc. Math. France, Paris, 2007.
	
	\bibitem{ISG99}
	Drago\c{s} Iftimie, Thomas~C. Sideris, and Pascal Gamblin.
	\newblock On the evolution of compactly supported planar vorticity.
	\newblock {\em Comm. Partial Differential Equations}, 24:1709--1730, 1999.
	
	\bibitem{Ki-sur2}
	Alexander Kiselev.
	\newblock Small scales and singularity formation in fluid dynamics.
	\newblock In {\em Proceedings of the {I}nternational {C}ongress of
		{M}athematicians---{R}io de {J}aneiro 2018. {V}ol. {III}. {I}nvited
		lectures}, pages 2363--2390. World Sci. Publ., Hackensack, NJ, 2018.
	
	\bibitem{KiLi}
	Alexander Kiselev and Chao Li.
	\newblock Global regularity and fast small-scale formation for {E}uler patch
	equation in a smooth domain.
	\newblock {\em Comm. Partial Differential Equations}, 44(4):279--308, 2019.
	
	\bibitem{KN}
	Alexander Kiselev and Fedor Nazarov.
	\newblock A simple energy pump for the surface quasi-geostrophic equation.
	\newblock In {\em Nonlinear partial differential equations}, volume~7 of {\em
		Abel Symp.}, pages 175--179. Springer, Heidelberg, 2012.
	
	\bibitem{KRYZ}
	Alexander Kiselev, Lenya Ryzhik, Yao Yao, and Andrej Zlato{\v{s}}.
	\newblock Finite time singularity for the modified {SQG} patch equation.
	\newblock {\em Ann. of Math. (2)}, 184(3):909--948, 2016.
	
	\bibitem{KYZ}
	Alexander Kiselev, Yao Yao, and Andrej Zlato\v{s}.
	\newblock Local regularity for the modified {SQG} patch equation.
	\newblock {\em Comm. Pure Appl. Math.}, 70(7):1253--1315, 2017.
	
	\bibitem{Ki-sur}
	Alexander~A. Kiselev.
	\newblock Small scale creation in active scalars.
	\newblock In {\em Progress in mathematical fluid dynamics}, volume 2272 of {\em
		Lecture Notes in Math.}, pages 125--161. Springer, Cham, [2020] \copyright
	2020.
	
	\bibitem{KX}
	R.~Krasny and L.~Xu.
	\newblock Vorticity and circulation decay in the viscous {L}amb dipole.
	\newblock {\em Fluid Dyn. Res.}, 53(015514), 2021.
	
	\bibitem{Lamb}
	H.~Lamb.
	\newblock {\em Hydrodynamics}.
	\newblock Cambridge Univ. Press., 3rd ed. edition, 1906.
	
	\bibitem{MB}
	Andrew~J. Majda and Andrea~L. Bertozzi.
	\newblock {\em Vorticity and incompressible flow}, volume~27 of {\em Cambridge
		Texts in Applied Mathematics}.
	\newblock Cambridge University Press, Cambridge, 2002.
	
	\bibitem{MV94}
	V.~V. Meleshko and G.~J.~F. van Heijst.
	\newblock On {C}haplygin's investigations of two-dimensional vortex structures
	in an inviscid fluid.
	\newblock {\em J. Fluid Mech.}, 272:157--182, 1994.
	
	\bibitem{Nad}
	N.~S. Nadirashvili.
	\newblock Wandering solutions of the two-dimensional {E}uler equation.
	\newblock {\em Funktsional. Anal. i Prilozhen.}, 25(3):70--71, 1991.
	
	\bibitem{Niel2}
	A.~H. Nielsen and J.~Juul Rasmussen.
	\newblock Formation and temporal evolution of the {L}amb-dipole.
	\newblock {\em Phys. Fluids}, 9(4):982--991, 1997.
	
	\bibitem{GeHe}
	J.~H. G.~M. van Geffen and G.~J.~F. van Heijst.
	\newblock Viscous evolution of 2{D} dipolar vortices.
	\newblock {\em Fluid Dynam. Res.}, 22(4):191--213, 1998.
	
	\bibitem{Xu}
	Xiaoqian Xu.
	\newblock Fast growth of the vorticity gradient in symmetric smooth domains for
	2{D} incompressible ideal flow.
	\newblock {\em J. Math. Anal. Appl.}, 439(2):594--607, 2016.
	
	\bibitem{Y3}
	V.~I. Yudovich.
	\newblock On the loss of smoothness of the solutions of the {E}uler equations
	and the inherent instability of flows of an ideal fluid.
	\newblock {\em Chaos}, 10(3):705--719, 2000.
	
	\bibitem{Zb}
	Samuel Zbarsky.
	\newblock From point vortices to vortex patches in self-similar expanding
	configurations.
	\newblock {\em arXiv:1912.10862}.
	
	\bibitem{Z}
	Andrej Zlato{\v{s}}.
	\newblock Exponential growth of the vorticity gradient for the {E}uler equation
	on the torus.
	\newblock {\em Adv. Math.}, 268:396--403, 2015.
	
\end{thebibliography}

\end{document}